\documentclass[11pt,letterpaper]{amsart}
\usepackage{amsfonts}
\usepackage{amsmath}
\usepackage{amsthm}
\usepackage{amssymb}
\usepackage{amsmath,amscd}
\usepackage{subcaption}
\usepackage[pdffitwindow=false,
colorlinks=true,linkcolor=black,urlcolor=black,citecolor=black]{hyperref}
\usepackage{tabularx}
\usepackage{float}

\newtheorem{theorem}{Theorem}[section]
\newtheorem{lemma}[theorem]{Lemma}
\theoremstyle{definition}

\newtheorem{condition}[theorem]{Condition}
\newtheorem{proposition}[theorem]{Proposition}

\numberwithin{equation}{section}
\newcommand{\spin}{\mathrm{Spin}}
\newcommand{\trace}{\mathrm{tr}}
\newcommand{\ricci}{\mathrm{Ric}}

\title{Non-existence of cohomogeneity one Einstein metrics of two summands}
\author{Hanci Chi}
\address{Department of Pure Mathematics\\ Xi'an Jiaotong-Liverpool University\\ Suzhou 215123\\ China}
\email{hanci.chi@xjtlu.edu.cn}
\begin{document}
\maketitle
\begin{abstract}
We prove the non-existence of cohomogeneity one Einstein metrics on a class of compact manifolds arising as double disk bundles, whose principal orbits split into two inequivalent irreducible summands. The proof uses a phase space barrier argument that yields a new obstruction to the existence of closed cohomogeneity one Einstein metrics, even when the principal orbits admit homogeneous Einstein metrics.
\end{abstract}
\let\thefootnote\relax\footnote{2020 Mathematics Subject Classification: 53C25 (primary).

Keywords: Einstein metric, cohomogeneity one metric.

The author is supported by the NSFC (No. 12301078), the Natural Science Foundation of Jiangsu Province (BK-20220282), and the XJTLU Research Development Funding (RDF-21-02-083).}

\section{Introduction}
A Riemannian metric $g$ is Einstein if $\ricci(g)=\Lambda g$ for some constant $\Lambda$. In dimension 4, it is well known that there are topological obstruction to Einstein metrics, see \cite{hitchin1974compact, thorpe1969some}. Seiberg--Witten theory yields further non‑existence theorem on symplectic 4‑manifolds that is not covered by the above theorems \cite{lebrun1996four}. In higher dimensions, while no general topological obstructions to Einstein metrics are known, certain geometric settings impose significant constraints. For instance, the existence of positive Kähler–Einstein metrics is tightly linked to stability conditions; see \cite{matsushima1957structure,lichnerowicz1958geometrie, futaki1983obstruction, tian1987kahler, tian2015k, chen2015kahler}.

Several non-existence criteria have been established in the context of manifolds with low cohomogeneity. For homogeneous Einstein metrics, see \cite{wang1986existence, park1997invariant, bohm2005non}. The non-existence theorem in \cite{wang1986existence} was later brought to the cohomogeneity one setting in \cite{bohm_inhomogeneous_1998}, and further generalized in \cite{bohm1999non}. These results emphasize the sign of the traceless part of the homogeneous Ricci tensor of the singular orbit, and yield infinite families of cohomogeneity one manifolds that do not admit any $\mathsf{G}$-invariant Einstein metrics. Necessarily, the principal orbits themselves do not admit any invariant Einstein metrics.

In this paper, we explore an alternative generalization to B\"ohm's non-existence theorem: we study the case where the principal orbit $\mathsf{G}/\mathsf{K}$ \emph{does} admit invariant Einstein metrics, and yet the double disk bundle fails to support any cohomogeneity one Einstein metric. Our setting follows that of \cite{chi2024existence}, where the principal orbit is the total space of a homogeneous fibration
\begin{equation}
\label{eqn_homogeneous fibration}
\mathbb{S}^{d_1}=\mathsf{H}/\mathsf{K}\hookrightarrow \mathsf{G}/\mathsf{K}\rightarrow \mathsf{G}/\mathsf{H}=\mathcal{Q}^{d_2},
\end{equation}
with both $\mathbb{S}^{d_1}$ and $\mathcal{Q}^{d_2}$ are irreducible. Hence, we have $\dim(\mathcal{M}^\mathsf{G})=2$, where $\mathcal{M}^\mathsf{G}$ is the space of $\mathsf{G}$-invariant metrics on $\mathsf{G}/\mathsf{K}$. The cohomogeneity one Einstein equation on the associated double disk bundle $M$ is a a system of ODEs characterized by the structural triple $(d_1,d_2,A)$. Here, the constant $A\geq 0$ quantifies the gap between $\mathsf{G}/\mathsf{K}$ and the product space $\mathbb{S}^{d_1}\times \mathcal{Q}^{d_2}$ \cite{o1966fundamental}. According to \cite[Theorem 3.1]{bohm_inhomogeneous_1998}, if
\begin{equation}
\label{eqn_bohm's lower bound for A}
A \geq \frac{1}{n+d_1} \frac{d_2(d_2 - 1)^2}{4(d_1 - 1)}, \quad n=d_1+d_2,
\end{equation}
the space $\mathcal{M}^\mathsf{G}$ contains at most one Einstein metric, and the double disk bundle $M$ admits none. Our main theorem is as follows.
\begin{theorem}
\label{thm_main}
Define
$$
\Psi_{d_1,d_2}:=\frac{(4 (d_1-1)n^2+d_2^2)(3n+d_1)}{(2n^2+n+d_1)^2 d_1^2}\frac{d_2(d_2-1)^2}{4(d_1-1)}.
$$
If $\mathsf{G}/\mathsf{K}$ is a principal orbit in \eqref{eqn_homogeneous fibration} with $d_1,d_2\geq 2$ and $d_1+d_2\geq 6$, $(d_1,d_2)\neq (2,4)$, and $A\in\left[\Psi_{d_1,d_2},\frac{1}{n+d_1}\frac{d_2(d_2-1)^2}{4(d_1-1)}\right)$, the double disk bundle $M$ does not admit any $\mathsf{G}$-invariant cohomogeneity one Einstein metric. 
\end{theorem}
 From the classification in \cite{dickinson_geometry_2008}, we have the following table of examples for Theorem \ref{thm_main}. The double disk bundle associated to $(\spin(7),\spin(8),\spin(9))$ is $\mathbb{OP}^2\sharp \overline{\mathbb{OP}}^2$. Unlike its counterparts $\mathbb{CP}^{m+1}\sharp \overline{\mathbb{CP}}^{m+1}$ and $\mathbb{HP}^{m+1}\sharp \overline{\mathbb{HP}}^{m+1}$, the existence theorem does not carry over.
\begin{table}[H]
\centering
\begin{tabular}{|l|l|l|l|l|l|l|}
\hline
$\mathsf{K}$&$\mathsf{H}$&$\mathsf{G}$& $d_1$& $d_2$& $A$&$\Psi_{d_1,d_2}$\\
\hline
\hline
&&&&&&\\
$Sp(2)U(1)$& $U(4)$ & $SU(5)$ & $5$ & $8$ & $\frac{49}{50}=0.98$& $\frac{186494}{198025}\approx 0.94$\\[2ex]

$\spin(7)$ & $\spin(8)$ & $\spin(9)$ &$7$ & $8$&$\frac{1}{2}$& $\frac{8879}{20886}\approx 0.43$\\[2ex]

$G_2\times SO(2)$ & $\spin(7)SO(2)$ &$\spin(9)$& $7$ & $14$&$\frac{507}{196}\approx 2.59$& $\frac{11}{6}\approx 1.83$\\[2ex]

$\spin(11)Sp(1)$ & $\spin(12)Sp(1)$ & $E_7$ & $11$ & $64$ & $49$ & $\frac{26823819708}{1214772845}\approx 22.08$ \\[2ex]

$\spin(15)$ & $\spin(16)$ & $E_8$& $15$ & $128$  & $\frac{32258}{225}\approx 143.37$ &$\frac{28882022881}{576131150}\approx 50.13$\\[2ex]
\hline
\end{tabular}
\caption{}
\label{table: Non-existence}
\end{table}

The theorem is also applicable for principal orbits of two summands type with $\dim(\mathcal{M}^\mathsf{G})\geq 3$, where the associated cohomogeneity one equations admit a subsystem with two diagonal metric components. 
\begin{table}[H]
\centering
\begin{tabular}{|l|l|l|l|l|l|l|}
\hline
$\mathsf{K}$&$\mathsf{H}$&$\mathsf{G}$& $d_1$& $d_2$&$A$&$\Psi_{d_1,d_2}$\\
\hline
\hline 
&&&&&&\\
$G_2$ & $\spin(7)$ &$\spin(8)$& $7$&$7$&$\frac{3}{8}$& $\frac{2037}{6962}\approx 0.29$\\[2ex]

$\spin(8)$& $\spin(9)$ & $F_4$ & $8$ & $16$ & $\frac{25}{9}\approx 2.78 $ &$\frac{18000}{9583}\approx 1.88$ \\[2ex]
\hline
\end{tabular}
\caption{}
\label{table: generalized Wallach spaces with A}
\end{table}
The two examples above have $\dim(\mathcal{M}^\mathsf{G})=3$ for different reasons.
For the first example, a homogeneous metric on $\mathsf{G}/\mathsf{K}$ has off-diagonal metric components as $\mathfrak{g}/\mathfrak{h}$ and $\mathfrak{h}/\mathfrak{k}$ are equivalent. For the second example, the isotropy representation split into three inequivalent irreducible summands, allowing $\mathsf{G}/\mathsf{H}$ to be squashed in the full system. Theorem \ref{thm_main} rules out the possiblility for the existence of cohomogeneity one Einstein metrics from the subsystem where $\mathsf{G}/\mathsf{K}$ is determined by two diagonal metric components.

This paper is structured as follow: In Section 2, we present the cohomogeneity one Einstein equation and recover \cite[Theorem 3.1]{bohm_inhomogeneous_1998}. In Section 3-4, we introduce a function $P$, which is also used for the existence theorem \cite{chi2024existence}, and construct a barrier set where the level surface $\{P=0\}$ is a face. The computations are carried out with the aid of MAPLE. For better readability, we leave those elementary yet complicated computations in the Appendix.

\section{B\"ohm's non-existence theorem}
\label{sec_Cohomogeneity one Einstein equations}
The dynamical system of cohomogeneity one Einstein metrics in this paper is identical to our earlier work in \cite{chi2024existence}. To avoid redundancy, we provide a concise summary of the ODE system in this section.

We consider the cohomogeneity one Einstein equations under the ansatz:
\begin{equation}
\label{eqn_ansatz}
dt^2+f_1^2(t)\left.b\right|_{\mathfrak{h}/\mathfrak{k}}+f_2^2(t)\left.b\right|_{\mathfrak{g}/\mathfrak{h}}.
\end{equation}
%is 
%\begin{align}
%\label{eqn: Einstein equation}
%\frac{\ddot{f_i}}{f_i}-\left(\frac{\dot{f_i}}{f_i}\right)^2&=-\trace(L)\frac{\dot{f_i}}{f_i}+r_i-\Lambda,\quad i=1,2;\\
%\label{eqn: conservation law in Einstein equation}
%d_1\frac{\ddot{f_1}}{f_1}+d_2\frac{\ddot{f_2}}{f_2}&=-\Lambda,
%\end{align}
%where
%$$
%r_1=\frac{d_1-1}{f_1^2}+A\frac{f_1^2}{f_2^4},\quad r_2=\frac{d_2-1}{f_2^2}-2\frac{d_1}{d_2} A\frac{f_1^2}{f_2^4}.
%$$
%The initial condition and terminal condition are
%\begin{equation}
%\label{eqn_initial condition}
%\lim\limits_{t\to 0}(f_1,\dot{f_1},f_2,\dot{f_2})=(0,1,f,0),\quad \lim\limits_{t\to t_*}(f_1,\dot{f_1},f_2,\dot{f_2})=(0,-1,\bar{f},0),\quad f,\bar{f}>0,
%\end{equation}
%so that the solution smoothly extends to both ends. 
%Complete solution on $M$ with $d_1=1$ was well studied in \cite{berard-bergery_sur_1982}. We further assume $d_2\geq d_1\geq 2$ so that $M$ is not a product space. 
Normalize the orbit space by $d\eta=\sqrt{\trace^2(L)+n\Lambda} dt$, where $L$ is the shape operator of $\mathsf{G}/\mathsf{K}$ in $M$. The cohomogeneity one Einstein equation for \eqref{eqn_ansatz} is transformed to a polynomial flow of $(X_1,X_2,Y,Z)$
\begin{equation}
\label{eqn_New Einstein equation}
\begin{split}
X_i'&=X_iH\left(G+\frac{1-H^2}{n}-1\right)+R_i-\frac{1-H^2}{n},\quad i=1,2;\\
Y'&=Y\left(H\left(G+\frac{1-H^2}{n}\right)-X_1\right),\\
Z'&=Z\left(H\left(G+\frac{1-H^2}{n}\right)+X_1-2X_2\right),\\
\end{split}
\end{equation}
on the invariant set 
\begin{equation}
\label{eqn_conservation equation for 1}
\mathcal{E}: \left\{\frac{1}{n-1}\left(G-H^2+d_1R_1+d_2R_2\right)=\frac{1-H^2}{n}\right\}\cap \{Y,Z\geq 0\}\cap \{H^2\leq 1\}, 
\end{equation}
where 
\begin{equation}
\label{eqn_curvature for G/K}
\begin{split}
&G:=d_1X_1^2+d_2X_2^2,\quad H:=d_1X_1+d_2X_2,\\
&R_1:=(d_1-1)Y^2+AZ^2,\quad R_2:=(d_2-1)YZ-\frac{2d_1}{d_2}AZ^2.
\end{split}
\end{equation}
The initial and terminal conditions where $\mathsf{G}/\mathsf{K}$ smoothly collapses to $\mathcal{Q}^{d_2}$ are respectively transformed to the following hyperbolic critical points 
$$p_0^\pm:=\left(\pm\frac{1}{d_1},0,\frac{1}{d_1},0\right).$$ 
By the local analysis in \cite{chi2024existence}, there exists a continuous $1$-parameter family of integral curves
$\{\gamma_s\mid s>0 \}$ that emanate $p_0^+$ and be in the interior of
\begin{equation}
\label{eqn_E+}
\mathcal{E}^+:=\mathcal{E}\cap \{H\geq 0\}\cap \{X_1-X_2\geq   0\}\cap \{\mu_1 Y-Z\geq 0\}
\end{equation}
Each integral curve represents a local non-homothetic solution on a tubular neighborhood around $\mathcal{Q}^{d_2}$. We aim to derive conditions under which each $\gamma_s$ does not converge to $p_0^-$.

We end this section by recovering B\"ohm's non-existence theorem in the following.

\begin{theorem}{\cite[Theorem 3.1]{bohm_inhomogeneous_1998}}
\label{thm_bohm's non-existence thm}
If \eqref{eqn_bohm's lower bound for A} holds, then $\lim\limits_{\eta\to\infty}\gamma_s\neq p_0^-$ for any $s>0$. 
\end{theorem}
\begin{proof}
Suppose $\gamma_s$ is a heterocline that joins $p_0^\pm$. Since $X_1-X_2$ is positive at $p_0^+$ and it is negative at $p_0^-$, we know that the function must vanish for the first time at some $\eta_*$. Then at $\gamma_s(\eta_*)$ we have
\begin{equation}
\label{eqn_deri of X1-X2}
(X_1-X_2)'=(X_1-X_2)H\left(G+\frac{1-H^2}{n}-1\right)+R_1-R_2=R_1-R_2.
\end{equation}
If $A>\frac{1}{n+d_1}\frac{d_2(d_2-1)^2}{4(d_1-1)}$, the quadratic equation
\begin{equation}
\label{eqn_homogeneous metric polynomial}
\frac{1}{Y^2}(R_1-R_2)=\frac{n+d_1}{d_2}A l^2-(d_2-1)l+(d_1-1)=0,\quad l:=\frac{Z}{Y}
\end{equation}
does not have any real root. We hence have $R_1>R_2$. The derivative $(X_1-X_2)'$ is positive at $\gamma_s(\eta_*)$, which is a contradiction.

If $A=\frac{1}{n+d_1}\frac{d_2(d_2-1)^2}{4(d_1-1)}$, it is necessary that $R_1-R_2$ also vanishes at $\gamma_s(\eta_*)$. Therefore, the point $\gamma_s(\eta_*)$ is in the invariant set $\mathcal{E}\cap \{X_1-X_2=0\}\cap \{R_1-R_2=0\}$. While $\gamma_s$ is not in this invariant set initially, we obtain a contradiction.
\end{proof}
Among examples of homogeneous spaces listed in \cite{dickinson_geometry_2008}, only 
$$\textbf{III}.11\quad (\mathsf{K},\mathsf{H},\mathsf{G})=(\spin(11),\spin(12),Sp(16))$$  forms a sphere bundle that satisfy \eqref{eqn_bohm's lower bound for A}. Starting from the next section, we assume $A\leq \frac{1}{n+d_1}\frac{d_2(d_2-1)^2}{4(d_1-1)}$ so that \eqref{eqn_homogeneous metric polynomial} always have two real roots $0<\mu_2\leq \mu_1$. 

\section{Barrier of non-existence: Part I}
\label{sec_Barrier of non-existence: Part I}
In our earlier work on existence theorem \cite{chi2024existence}, we introduced a barrier set in
$\mathcal{E}^+$
to prove the existence of a $\gamma_s$ that joins $p_0^\pm$. The following polynomial plays an important role in the construction.
\begin{equation}
\label{eqn_barrier functions}
\begin{split}
P&:=X_1\left(R_2-\frac{1-H^2}{n}\right)-X_2\left(R_1-\frac{1-H^2}{n}\right)-2X_2(X_1-X_2)\left(X_1+\frac{d_2}{2d_1}X_2\right).
\end{split}
\end{equation}
Along an integral curve, we have 
\begin{align}
\label{eqn_P'}
P'&=P\left(H\left(3G+\frac{3}{n}(1-H^2)-1\right)+\frac{n}{d_1}X_2-X_1\right)+(X_1-X_2)(P+Q),
\end{align}
where 
\begin{equation}
\label{eqn_function Q}
\begin{split}
Q&=4X_2\left(X_1+\frac{d_2}{2d_1}X_2\right)\left(H+\frac{d_2}{2d_1}X_2\right)+\left(2X_1+2X_2+\frac{3d_2}{d_1}X_2\right)\frac{1-H^2}{n}\\
&\quad -2(d_2-1)X_1YZ-X_2\left(2(d_1-1)Y^2+\frac{3d_2}{d_1}(d_2-1)YZ\right).
\end{split}
\end{equation}
We showed that for $A>0$ sufficiently small, the inequality $P>0$ holds along some $\gamma_s$ and forces the integral curve to cross the face $\{X_1-X_2=0\}$, which leads to the existence of a cohomogeneity one Einstein metric on $M$. Here, we show that for an $A$ sufficiently large, the inequality $P<0$ holds along \emph{any} $\gamma_s$ until the integral curve cross the face $\{X_1-X_2>0\}\cap \{H=0\}$. 

Note that the last defining inequality in $\mathcal{E}^+$ is motivated by an elementary observation below.
\begin{proposition}
\label{prop_l gets too large}
If a $\gamma_s$ enters the set 
\begin{equation}
\label{eqn_bad set}
\mathcal{E} \cap \{X_1-X_2\geq 0\}\cap \{Z-\mu_1 Y\geq 0\},
\end{equation}
 then the integral curve does not converge to $p_0^-$.
\end{proposition}
\begin{proof}
We have
\begin{equation}
\begin{split}
\left(\frac{Z}{Y}\right)'&=2\frac{Z}{Y}(X_1-X_2)\\
(X_1-X_2)'&=(X_1-X_2)H\left(G+\frac{1}{n}(1-H^2)-1\right)+R_1-R_2.
\end{split}
\end{equation}
Since $R_1-R_2\geq  0$ if $\frac{Z}{Y}\geq \mu_1$,  the set \eqref{eqn_bad set} is invariant. As $X_1-X_2$ is negative at $p_0^-$, a $\gamma_s$ does not converge to $p_0^-$ if it enters the set \eqref{eqn_bad set}.
\end{proof}
Hence, it suffices to construct the barrier set of the non-existence in $\mathcal{E}^+$. Consider the interior $\left(\mathcal{E}^+\right)^\circ$ as the union
$$
\left(\mathcal{E}^+\right)^\circ = \bigsqcup_{(k,l)\in\left(-\frac{d_1}{d_2},1\right)\times (0,\mu_1)} \mathcal{E}_{k,l},
$$
where
\begin{equation}
\label{eqn_kl slice}
\begin{split}
\mathcal{E}_{k,l}&:= \mathcal{E}^+ \cap \{X_2-kX_1=0,X_1> 0\}\cap \{Z-lY=0, Y> 0\}.
\end{split}
\end{equation}

Define
$$
\mathfrak{C}_A=\left(\mathcal{E}^+\right)^\circ\cap \{P< 0\}.
$$
Homogenize $P$ and $Q$ by replacing the term $\frac{1-H^2}{n}$ in the polynomials by the first equation in \eqref{eqn_conservation equation for 1}. We parametrize $\left(\mathcal{E}^+\right)^\circ \cap \{P=0\}$ using $(k,l)$ as in \eqref{eqn_kl slice}. On each slice $\mathcal{E}_{k,l}$, we have 
\begin{equation}
\label{eqn_P and Q in X and Y}
\begin{split}
P&=P_Y Y^2X_1+P_X X_1^3,\quad Q=Q_Y Y^2X_1+Q_X X_1^3,
\end{split}
\end{equation}
where
\begin{equation}
\label{eqn_PY}
\begin{split}
P_Y(A,k,l)&=-\frac{d_1(n+d_1-2)+d_2(n+d_1-1)k}{d_2 (n-1)} Al^2 + \frac{(d_2-1)(d_1+d_2k-1)}{n-1}l\\
&\quad -\frac{(d_1-1)(d_1+d_2 k-k)}{n-1},
\end{split}
\end{equation}
\begin{equation}
\label{eqn_PX}
\begin{split}
P_X(k)&=\frac{d_2(d_1d_2 - 2d_1 - d_2 +1)k^2 + 2(d_2 - 1)(d_1 -1)d_1k + d_1^2(d_1-1)}{d_1(n - 1)}(1 - k),
\end{split}
\end{equation}
\begin{equation}
\label{eqn_QY}
\begin{split}
Q_Y(A,k,l)&=-\frac{2 n k+d_2 k+2 d_1}{n-1}Al^2-\frac{2d_1(d_1-1)+d_1d_2 k-3d_2 k}{d_1(n-1)}(d_2-1)l\\
&\quad +\frac{(2d_1+d_2 k+2 k)(d_1-1)}{n-1},
\end{split}
\end{equation}

\begin{equation}
\label{eqn_QX}
\begin{split}
Q_X(k)&=4k\left(1+\frac{d_2 k}{2d_1}\right)\left(d_1+d_2k+\frac{d_2k}{2d_1}\right)+\left(2+2k+\frac{3d_2k}{d_1}\right)\frac{d_1+d_2k^2-(d_1+d_2k)^2}{n-1}.
\end{split}
\end{equation}

\begin{proposition}
\label{prop_px is negative somewhere}
For any $d_1,d_2\geq 2$, the polynomial $P_X$ vanishes at some $k_*\in\left(-\frac{d_1(n+d_1-2)}{d_2(n+d_1-1)},-\frac{d_1-1}{d_2}\right)$. The polynomial is negative at $\left[-\frac{d_1}{d_2},k_*\right)$ and is positive at $(k_*,1)$.
\end{proposition}
\begin{proof}
We have $P_X=\frac{1-k}{d_1(n-1)}\tilde{P}_X$, where
$$
\tilde{P}_X:=d_2(d_1d_2 - 2d_1 - d_2 +1)k^2 + 2(d_2 - 1)(d_1 -1)d_1k + d_1^2(d_1-1).
$$
We have
$$
\tilde{P}_X\left(-\frac{d_1}{d_2}\right)=-\frac{d_1^2}{d_2}<0,\quad
\tilde{P}_X\left(-\frac{d_1(n+d_1-2)}{d_2(n+d_1-1)}\right)= -\frac{d_1^2(n-1)}{(n+d_1-1)^2}<0,
$$
$$ \tilde{P}_X\left(-\frac{d_1-1}{d_2}\right)= \frac{(d_1-1)(n-1)}{d_2}>0.
$$
The polynomial $\tilde{P}_X$ vanishes at some $k_*\in\left(-\frac{d_1(n+d_1-2)}{d_2(n+d_1-1)},-\frac{d_1-1}{d_2}\right)$. If $d_2>2$, the polynomial $\tilde{P}_X$ is concave up. The evaluations above indicate that $k_*$ is the larger negative root of $\tilde{P}_X$, and the smaller root must be less than $-\frac{d_1}{d_2}$. If $d_2=2$, the polynomial $\tilde{P}_X$ is concave down. From $\tilde{P}_X(1)=(n-1)(d_1n-n-d_1)>0$ it is clear that $k_*$ is the smaller root of $\tilde{P}_X$, and the larger root is larger than 1. For any $d_1\geq 2$, the factor $\tilde{P}_X$ is negative on $\left[-\frac{d_1}{d_2},k_*\right)$ and is positive on $(k_*,1)$. The proof is complete.
\end{proof}

We use $k_*$ in Proposition \ref{prop_px is negative somewhere} to divide $(\mathcal{E}^+)^\circ$ into the following three subsets.
$$
\mathcal{E}^+_{2}:=\bigcup\limits_{(k,l)\in\left (-\frac{d_1}{d_2},k_*\right]\times (0,\mu_1)}\mathcal{E}_{k,l},\quad \mathcal{E}^+_{1}:=\bigcup\limits_{(k,l)\in\left (k_*,0\right]\times (0,\mu_1)}\mathcal{E}_{k,l}, \quad \mathcal{E}^+_{0}:=\bigcup\limits_{(k,l)\in (0,1)\times (0,\mu_1)}\mathcal{E}_{k,l}.
$$
We consider the sign of $P$ and $Q$ on each $\mathcal{E}^+_i\cap \mathfrak{C}_A$.

\begin{proposition}
\label{prop_P<0 on E2}
For any $A\geq 0$. The inequality $P<0$ holds on $\mathcal{E}^+_{2}.$
\end{proposition}
\begin{proof}
By Proposition \ref{prop_px is negative somewhere}, the inequality
\begin{equation}
\begin{split}
P&=P_YY^2X_1+P_X X_1^3\leq P_YY^2X_1
\end{split}
\end{equation}
holds on $\mathcal{E}^+_2$.
It suffices to show that $P_Y(A,k,l)$ is negative for any $(k,l)\in\left(-\frac{d_1}{d_2},k_*\right]\times (0,\mu_1)$. We have 
\begin{equation}
\label{eqn_PY at l=0}
P_Y(A,k,0)=-\frac{(d_1-1)(d_1+kd_2-k)}{n-1}, \quad \frac{\partial P_Y}{\partial l}(A,k,0)=\frac{(d_1-1)(d_1+d_2k-1)}{n-1}.
\end{equation}
The first quantity in \eqref{eqn_PY at l=0} is negative for any $k\in \left(-\frac{d_1}{d_2},0\right]$. Since $k_*<-\frac{d_1-1}{d_2}$ by Proposition \ref{prop_px is negative somewhere}, the second quantity in \eqref{eqn_PY at l=0} is also negative for any $k\in \left(-\frac{d_1}{d_2},k_*\right]$. As a function of $l$, the polynomial $P_Y(A,k,l)$ either decreases first and then increases or monotonically decreases on $(0,\infty)$. Therefore, it suffices to check the negativity of $P_Y(A,k,\mu_1)$. Since $R_1=R_2$ if $l=\mu_1$, we have 
\begin{equation}
\begin{split}
P_Y(A,k,\mu_1)Y^2&=X_1R_2-kX_1R_1-\frac{d_1R_1+d_2R_2}{n-1}(X_1-kX_1)\\
&=X_1R_1-kX_1R_1-\frac{n}{n-1}R_1(X_1-kX_1)\\
&
=-\frac{1-k}{n-1}R_1< 0.
\end{split}
\end{equation}
Therefore, the inequality $P_Y<0$ holds for any $(k,l)\in\left(-\frac{d_1}{d_2},k_*\right]\times (0,\mu_1)$. The proof is complete.
\end{proof}
Proposition \ref{prop_P<0 on E2} implies that $\left(\mathcal{E}^+\right)^\circ \cap \{P=0\}$ does not intersect $\mathcal{E}^+_{2}$. To investigate the condition on $A$ for making $\mathcal{E}^+\cap \{P=0\}$ a barrier, we need to investigate the sign of $Q$ on $(\mathcal{E}^+_{1}\cup \mathcal{E}^+_0)\cap \{P=0\}$.

\section{Barrier of non-existence: Part II}
\label{sec_Barrier of non-existence: Part II}
With $P=0$, the function $Q$ can be rewritten as
\begin{equation}
\label{eqn_Q and omegaA}
\begin{split}
Q&= Q_Y Y^2X_1+Q_X \frac{P-P_YY^2X_1}{P_X}=\frac{Q_YP_X-Q_XP_Y}{P_X}Y^2X_1.
\end{split}
\end{equation}
Define the polynomial function $\omega(A,k,l)=Q_YP_X-Q_XP_Y$. Since $P_X>0$ on $\mathcal{E}^+_{1}\cup \mathcal{E}^+_0$, the function $\omega$ has the same sign as $Q$ on $(\mathcal{E}^+_{1}\cup \mathcal{E}^+_0)\cap \{P=0\}$. 
We have
\begin{equation}
\label{eqn_omega formula}
\begin{split}
\omega(A,k,l)&=\frac{1}{d_1^2(n-1)}\left(\frac{2d_1+d_2k}{d_2}\omega_2 Al^2+(d_2-1)k\omega_1 l -(d_1-1)k^2\omega_0\right),
\end{split}
\end{equation}
where 
\begin{equation}
\begin{split}
\omega_2(k)&=(2d_1^2d_2^2 - d_1d_2^3 + d_2^3 - d_2^2)k^3 + (4d_1^3d_2 - 4d_1^2d_2^2 - 2d_1^2d_2 + 4d_1d_2^2 - 2d_1d_2)k^2\\
&\quad + (2d_1^4 - 5d_1^3d_2 - 2d_1^3 + 5d_1^2d_2)k - 2d_1^4 + 2d_1^3,\\
\omega_1(k)&=(d_1d_2^3 - 4d_1d_2^2 - d_2^3 + 3d_2^2)k^3 + (2d_1^2d_2^2 - 8d_1^2d_2 + 8d_1d_2 - 2d_2^2)k^2 \\
&\quad + (d_1^3d_2 - 4d_1^3 + 5d_1^2d_2 + 4d_1^2 - 6d_1d_2)k + 4d_1^3 - 4d_1^2,\\
\omega_0(k)&=(d_1d_2^3 - 2d_1d_2^2 - d_2^3 - 2d_1d_2 + d_2^2)k^2 \\
&\quad + (2d_1^2d_2^2 - 2d_1^2d_2 - 2d_1d_2^2 - 4d_1^2 + 4d_1d_2)k + d_1^3d_2 - d_1^2d_2 + 4d_1^2.
\end{split}
\end{equation}

By \cite[Proposition 4.8]{chi2024existence}, the inequalities $\omega_2<0$ and $\omega_0>0$ hold on $[0,1]$. These inequalities can be extended to $\left[-\frac{d_1}{d_2},1\right]$; see Proposition \ref{prop_omega2} and Proposition \ref{prop_omega0} in the Appendix.

The $(k,l)$-coordinate of $p_0^+$ is $(0,0)$. By \cite[Proposition 4.8]{chi2024existence}, the function $\omega(A,k,l)$ attains a local maximum at $(0,0)$, provided that
\begin{equation}
\label{eqn_local barrier estimate}
A>\frac{d_2(d_2-1)^2}{d_1^2(d_1d_2-d_2+4)}. 
\end{equation}
In other words, with \eqref{eqn_local barrier estimate} satisfied, the face $\left(\mathcal{E}^+\right)^\circ \cap \{P=0\}$ is locally a barrier for $\mathfrak{C}_A$ near $p_0^+$. It is expected to increase the value of $A$ to make $\left(\mathcal{E}^+\right)^\circ \cap \{P=0\}$ a global barrier. An attempt to achieve this is considering the following condition:
\begin{condition}
\label{con_attempt 1}
As a quadratic of $l$, the function $\omega(A,k,l)$ does not have any real root for any $k\in (k_*,1)$.
\end{condition}
Note that as $k\to 1$, the polynomial $\omega$ is essentially the homogenized $R_2-R_1$. More explicitly,
The discriminant of $\omega(A,k,l)$ is
$$
\frac{k^2}{d_1^4(n-1)^2}\left((d_2-1)^2\omega_1^2+4(d_1-1)\frac{2d_1+d_2k}{d_2}\omega_0\omega_2A\right).
$$
Define
\begin{equation}
\label{eqn_Omega}
\Omega_{d_1,d_2}(k):=-\frac{\omega_1^2}{(2d_1+d_2k)\omega_0\omega_2}\frac{d_2(d_2-1)^2}{4(d_1-1)}.
\end{equation}
We have 
\begin{equation}
\label{eqn_max_Omega}
\Omega_{d_1,d_2}(1)=\frac{1}{n+d_1}\frac{d_2(d_2-1)^2}{4(d_1-1)}.
\end{equation}
Therefore, Condition \ref{con_attempt 1} is too strong to generate any new estimate for $A$ other than the one in \cite[Theorem 3.1]{bohm_inhomogeneous_1998}. 

Since $P_X>0$ for any $k\in (k_*,1)$, the vanishing of $P$ implies $P_Y(A,k,l)<0$ for any $k\in (k_*,1)$. We hence consider the following weaker condition:
\begin{condition}
\label{con_attempt 2}
As a quadratic of $l$, the inequality $\omega(A,k,l)< 0$ holds on $\{l\mid l>0, P_Y(A,k,l)<0\}$ for any $k\in (k_*,1)$.
\end{condition}

To obtain an estimate of $A$ that satisfy Condition \ref{con_attempt 2}, it is natural to consider the resultant between $\omega$ and $P_Y$ as functions of $l$.
\begin{proposition}
\label{prop_resultant py qy}
For any $A>0$, the polynomials $\omega$ and $P_Y$ do not share any common root for any $k\in (k_*,0]$.  If $A\geq \Psi_{d_1,d_2}$, then $\omega$ and $P_Y$ do not share any common root for any $k\in (0,1)$.
\end{proposition}
\begin{proof}
We compute 
\begin{equation}
\label{eqn_res_omega_PY}
\begin{split}
\mathrm{Res}_l(\omega,P_Y)&=\mathrm{Res}_l(P_XQ_Y-P_YQ_X,P_Y)\\
%&=\mathrm{Res}_l(P_XQ_Y,P_Y)\\
&=P_X^2\mathrm{Res}_l(Q_Y,P_Y)\\
&=P_X^2\frac{(d_1-1)A}{d_1^2d_2^2(n-1)^2}(\rho_1A-\rho_0),
\end{split}
\end{equation}
where 
\begin{equation}
\begin{split}
\rho_1&:=4 d_1^2 (d_1-1)(2d_2^2 k^2+ d_2k^2 + 4 d_1 d_2k+ 2d_1k  + 2d_1^2)^2,\\
\rho_0&:=(d_2-1)^2kd_2(4d_1+3d_2 k)(4(d_1-1)(d_1+d_2k)^2+d_2^2k^2).\\
%&=(d_2-1)^2kd_2(4d_1+3d_2 k)(4(d_1-1)(d_1+d_2k)^2+d_2^2k^2).
\end{split}
\end{equation}
By Proposition \ref{prop_px is negative somewhere}, the factor $P_X$ does not vanish on $(k_*,1)$. By observation, it is obvious that $\rho_1A-\rho_0>0$ for any $(A,k)\in (0,\infty)\times (k_*,0]$. Therefore, the quadratic functions $\omega$ and $P_Y$ do not share any common root for each $(A,k)\in (0,\infty)\times (k_*,0]$.

Assume $k\in (0,1)$ in the following. As $\rho_1>0$ for $k\in (0,1)$, we consider the rational function $\frac{\rho_0}{\rho_1}$ and compute
$$
\frac{d}{dk}\left(\frac{\rho_0}{\rho_1}\right)=\frac{2d_2(d_2-1)^2(d_1+d_2k)}{(d_1-1)(2d_2^2 k^2+ d_2k^2 + 4 d_1 d_2k+ 2d_1k  + 2d_1^2)^3}\rho_3,
$$
where 
\begin{equation}
\begin{split}
\rho_3&=(2  d_2^3 +  d_2^2) k^3 + (8 d_1  d_2^2 - 2 d_1  d_2 - 5  d_2^2 + 2  d_2) k^2\\
&\quad  + (10 d_1^2  d_2 - 4 d_1^2 - 10 d_1  d_2 + 4 d_1) k + 4 d_1^3 - 4 d_1^2.\\
%&= (2  d_2^3 +  d_2^2) k^3 + (3d_2^2+2d_2(4d_2-1)(d_1 -1)) k^2\\
%&\quad  + (10  d_2- 4)d_1(d_1-1)  k + 4 d_1^2(d_1-1)\\
%&>0
\end{split}
\end{equation}
Since the coefficient of each $k^i$ in $\rho_3$ is positive, the polynomial is positive for each $k\in (0,1)$. The function $\frac{\rho_0}{\rho_1}$ monotonically increases on $(0,1)$. Take 
$$
 \Psi_{d_1,d_2}:=\left(\frac{\rho_0}{\rho_1}\right)(1)=\frac{(4 (d_1-1)n^2+d_2^2)(3n+d_1)}{(2n^2+n+d_1)^2 d_1^2}\frac{d_2(d_2-1)^2}{4(d_1-1)}.
$$
It is clear that $\mathrm{Res}_l(\omega_A,P_Y)> 0$ for each $k\in (0,1)$ if $A\geq \Psi_{d_1,d_2}$. The proof is complete.
\end{proof}

We proceed to prove that $\omega< 0$ on $\mathcal{E}^+_{1}\cap \{P=0\}$.
\begin{proposition}
\label{prop_roots of omega make Py>0}
The quadratic function $P_Y(A,\tilde{k},l)$ is concave down for any $(A,k)\in (0,\infty)\times (k_*,0)$. Suppose the polynomial $\omega(\tilde{A},\tilde{k},l)$ has positive real roots $\mu_1(\tilde{A})\geq \mu_2(\tilde{A})>0$ for some $(\tilde{A},\tilde{k})\in (0,\infty)\times (k_*,0)$. Then $P_Y(\tilde{A},\tilde{k},l)$ also have real roots $\psi_1(\tilde{A})> \psi_2(\tilde{A})$ and the inclusion $[\mu_2(\tilde{A}),\mu_1(\tilde{A})]\subset (\psi_2(\tilde{A}),\psi_1(\tilde{A}))$ holds.
\end{proposition}
\begin{proof}
Since $k_*>-\frac{d_1(n+d_1-2)}{d_2(n+d_1-1)}$ by Proposition \ref{prop_px is negative somewhere}, the polynomial $P_Y(A,k,l)$ is concave down for any $(A,k)\in (0,\infty)\times (k_*,1]$ by \eqref{eqn_PY}.

By Proposition \ref{prop_omega2} and Proposition \ref{prop_omega0} in the Appendix, if $\omega(\tilde{A},\tilde{k},l)$ has positive real roots for some $\tilde{k}\in (k_*,0)$, we must have 
$$\omega_1(\tilde{k})<0.$$ With such a $\tilde{k}$ fixed, the polynomial $\omega(A,\tilde{k},l)$ has positive real roots $\mu_2(A)<\mu_1(A)$ for any $A\in (0,\tilde{A})$.  It is clear that $\lim\limits_{A\to 0}\mu_2(A)=\frac{(d_1-1)\tilde{k}\omega_0(\tilde{k})}{(d_2-1)\omega_1(\tilde{k})}$. 

Let
$$
P_Y(A,k,l)=P_{Y,2}(k)Al^2+P_{Y,1}(k)l+P_{Y,0}(k),\quad Q_Y(A,k,l)=Q_{Y,2}(k)Al^2+Q_{Y,1}(k)l+Q_{Y,0}(k).
$$
By the definition of $\omega$, we have
\begin{equation}
\label{eqn_omega1 and omega0}
\begin{split}
-\frac{(d_1-1)k^2}{d_1^2(n-1)}\omega_0=Q_{Y,0}P_X-P_{Y,0}Q_X,\quad \frac{(d_2-1)k}{d_1^2(n-1)}\omega_1=Q_{Y,1}P_X-P_{Y,1}Q_X.
\end{split}
\end{equation}
Therefore,
\begin{equation}
\label{eqn_P_Y(mu_2)}
\begin{split}
\lim\limits_{A\to 0} P_Y(A,\tilde{k},\mu_2(A))
%&=P_{Y,1}\frac{(d_1-1)\tilde{k}\omega_0}{(d_2-1)\omega_1}+P_{Y,0}\\
&=\frac{1}{(d_2-1)\omega_1(\tilde{k})}\left(P_{Y,1}(d_1-1)\tilde{k}\omega_0(\tilde{k})+P_{Y,0}(d_2-1)\omega_1(\tilde{k})\right)\\
&=\frac{d_1^2(n-1)}{(d_2-1)\tilde{k}\omega_1(\tilde{k})}P_X (P_{Y,0} Q_{Y,1}-P_{Y,1} Q_{Y,0}) \quad \text{by \eqref{eqn_omega1 and omega0}}\\
&=-\frac{d_1(d_1-1)(4d_1+3d_2\tilde{k})}{\omega_1(\tilde{k})}P_X(\tilde{k})\\
&> 0.
\end{split}
\end{equation}
As $P_Y$ is concave down, the limit \eqref{eqn_P_Y(mu_2)} indicates that for a sufficiently small $\underline{A}>0$, not only the real roots $\psi_2(\underline{A})<\psi_1(\underline{A})$ of $P_Y(\underline{A},\tilde{k},l)$ exist, but also 
\begin{equation}
\label{eqn_inclusion undreline A}
\mu_2(\underline{A})\in (\psi_2(\underline{A}),\psi_1(\underline{A})).
\end{equation}

Suppose $P_Y(\tilde{A},\tilde{k},l)$ does not have real roots. Choose $\check{A}\in (\underline{A},\tilde{A})$ such that $\psi_2(\check{A})=\psi_1(\check{A})$. Proposition \ref{prop_resultant py qy} is violated by some $A\in (\underline{A},\check{A}]$ since \eqref{eqn_inclusion undreline A} holds. Therefore, the quadratic function $P_Y(\tilde{A},\tilde{k},l)$ must have real roots $\psi_1(\tilde{A})>\psi_2(\tilde{A})$ and $\mu_2(\tilde{A})\in (\psi_2(\tilde{A}),\psi_1(\tilde{A}))$.

We claim that it is also true that $\mu_1(\tilde{A})\in (\psi_2(\tilde{A}),\psi_1(\tilde{A}))$. Suppose otherwise; then we must have 
$$\psi_2(\tilde{A})<\mu_2(\tilde{A})<\psi_1(\tilde{A})<\mu_1(\tilde{A}).$$
Let $\hat{A}$ be the smallest number in $[\tilde{A},\infty)$ such that at least one of $\mu_2(\hat{A})=\mu_1(\hat{A})$ and $\psi_2(\hat{A})=\psi_1(\hat{A})$ is satisfied. Proposition \ref{prop_resultant py qy} is violated by some $A\in (\tilde{A},\hat{A}]$. Therefore, the inclusion $[\mu_2(\tilde{A}),\mu_1(\tilde{A})]\subset (\psi_2(\tilde{A}),\psi_1(\tilde{A}))$ must hold.
\end{proof}

\begin{proposition}
\label{prop_Q<0 on E1}
If $A>0$, the function $\omega$ is negative on $\mathcal{E}^+_{1}\cap \{P=0\}$.
\end{proposition}
\begin{proof}
If $k=0$, then $\omega(A,0,l)< 0$ for any $l\in (0,\infty)$.

Fix a $k\in (k_*,0)$ in the following. If $\omega(A,k,l)$ does not have real roots, then the function is negative for any $l\in\mathbb{R}$. Note that this case includes $\omega_1(k)=0$. If $\omega(A,k,l)$ has real roots and $\omega_1(k)>0$, both roots are negative and the function is negative for any $l\in (0,\infty)$. If $\omega_1(k)<0$, then both roots of $\omega(A,k,l)$ are positive. Proposition \ref{prop_roots of omega make Py>0} implies that $\omega(A,k,l)<0$ on $\{l\mid l>0, P_Y(A,k,l)<0\}$. By the arbitrariness of $k\in (k_*,0]$, we conclude that $\omega<0$ on $\mathcal{E}^+_{1}\cap \{P=0\}$.
\end{proof}

We consider the sign of $\omega$ on $\mathcal{E}^+_0\cap \{P=0\}$.
\begin{proposition}
\label{prop_prop of Omega 2}
Let $(d_1,d_2)\notin \{(2,2),(3,2)\}$. For a fixed $A\in \left[\Psi_{d_1,d_2},\frac{1}{n+d_1}\frac{d_2(d_2-1)^2}{4(d_1-1)}\right)$, if $\omega(A,\tilde{k},l)$ has real roots for some $\tilde{k}\in (0,1)$, then so does $\omega(A,k,l)$ for any $k\in [\tilde{k},1)$.
\end{proposition}
\begin{proof}
Since $\omega(A,\tilde{k},l)$ has real roots, it is clear that $\Omega_{d_1,d_2}(\tilde{k})\geq A$. From Proposition \ref{prop_property of Omega} in the Appendix, the function $\Omega_{d_1,d_2}$ either monotonically increases or first decreases and then increases on $(0,1)$. As
\begin{equation}
\label{eqn_Omega and Psi}
\Omega_{d_1,d_2}(0)=\frac{4(d_1-1)}{d_1^2(d_1d_2-d_2+4)}\frac{d_2(d_2-1)^2}{4(d_1-1)}\leq \Psi_{d_1,d_2},
\end{equation}
there exists a unique $k_\bullet\in(0,1)$ such that $\Omega_{d_1,d_2}(k_\bullet)=\Psi_{d_1,d_2}$ and $\Omega_{d_1,d_2}$ monotonically increases on $(k_\bullet,1)$. By $\Omega_{d_1,d_2}(\tilde{k})\geq A\geq \Psi_{d_1,d_2}=\Omega_{d_1,d_2}(k_\bullet)$, we must have $\tilde{k}\in [k_\bullet,1)$. Furthermore, we have 
$$\Omega_{d_1,d_2}(k)\geq \Omega_{d_1,d_2}(\tilde{k})\geq A$$ for any $k\in [\tilde{k},1)$. Therefore, the discriminant of $\omega(A,k,l)$ is non-negative for any $k\in (\tilde{k},1)$. The proof is complete.
\end{proof}

\begin{proposition}
\label{prop_initial inclusion}
Let $(d_1,d_2)\notin\{(2,2),(2,3),(2,4)\}$. For any $A\in \left[\Psi_{d_1,d_2},\frac{1}{n+d_1}\frac{d_2(d_2-1)^2}{4(d_1-1)}\right)$, there exists a sufficiently small $\delta>0$ so that $\omega(A,1-\delta,l)$ have real roots $\mu_2(A,1-\delta)<\mu_1(A,1-\delta)$ and $P_Y(A,1-\delta,l)$ have real roots $\psi_2(A,1-\delta)< \psi_1(A,1-\delta)$. Furthermore, we have the inclusion  $[\mu_2(A,1-\delta), \mu_1(A,1-\delta)]\subset (\psi_2(A,1-\delta), \psi_1(A,1-\delta))$.
\end{proposition}
\begin{proof}
%From Proposition \ref{prop_property of Omega}, the function $\Omega_{d_1,d_2}$ either monotonically increases or it decreases and then increases. As
%$$\Omega_{d_1,d_2}(0)=\frac{4(d_1-1)}{d_1^2(d_1d_2-d_2+4)}\frac{d_2(d_2-1)^2}{4(d_1-1)}\leq \Psi_{d_1,d_2},$$ there exists a unique $k_\bullet\in(0,1)$ such that $\Omega_{d_1,d_2}(k_\bullet)=\Psi_{d_1,d_2}$ and $\Omega$ monotonically increases on $(k_\bullet,1)$. Then with $A\in \left[\Psi_{d_1,d_2},\frac{1}{n+d_1}\frac{d_2(d_2-1)^2}{4(d_1-1)}\right)$, there exists a $\tilde{k}\in [k_\bullet,1)$ such that each $\omega(A,k,l)$ with $k\in [\tilde{k},1]$ has real roots.

For each fixed $k\in (0,1)$, the polynomial $P_Y$ is concave down and $P_Y(A,k,0)<0$. Compute the discriminant for $P_Y$. The polynomial $P_Y$ has real roots if $A\leq \Xi_{d_1,d_2}(k)$, where
$$
\Xi_{d_1,d_2}(k):=\frac{d_2(d_2-1)^2(d_1+d_2k-1)^2}{4(d_1-1)(d_1+kd_2-k)(d_1(n+d_1-2)+d_2(n+d_1-1)k)}.
$$
We have $\Xi_{d_1,d_2}(1)=\Omega_{d_1,d_2}(1)=\frac{1}{n+d_1}\frac{d_2(d_2-1)^2}{4(d_1-1)}$ and 
$$\frac{d \Xi_{d_1,d_2}}{dk}(1)=\frac{d_2n(d_2-1)^2}{2(d_1-1)(n-1)(2d_1+d_2)^2},\quad \frac{d\Omega_{d_1,d_2}}{dk}(1)=\frac{n(d_2-1)^2(d_1^2+d_1d_2-3d_1-d_2)}{(d_1-1)(n-1)(2d_1+d_2)^2}.$$
If $(d_1,d_2)\notin\{(2,2),(2,3),(2,4)\}$, then $\frac{d\Omega_{d_1,d_2}}{dk}(1)>\frac{d\Xi_{d_1,d_2}}{dk}(1)$. 
%While $\frac{d \Xi_{2,4}}{dk}(1)= \frac{d\Omega_{2,4}}{dk}(1),$ we have 
%$$
%\frac{d^2 \Xi_{2,4}}{dk^2}(1)=-\frac{153}{400},\quad  \frac{d^2\Omega_{2,4}}{dk^2}(1)=\frac{477}{400}.
%$$
Hence, for these dimensions, there exists a small enough $\delta>0$ such that 
$$A<\Omega_{d_1,d_2}(k)< \Xi_{d_1,d_2}(k)<\Omega_{d_1,d_2}(1)=\Xi_{d_1,d_2}(1)$$ for any $k\in [1-\delta,1]$. Therefore, both $\mu_i(A,k)$ and $\psi_i(A,k)$ exist for any $k\in [1-\delta,1]$.

There are six possibilities for $\mu_i(A,1-\delta)$ and $\psi_i(A,1-\delta)$.
$$
\mu_2<\psi_2\leq \psi_1<\mu_1,\quad \mu_2<\psi_2<\mu_1<\psi_1,\quad \psi_2<\mu_2<\psi_1<\mu_1,
$$
$$
\quad \mu_2\leq \mu_1<\psi_2\leq \psi_1,\quad \psi_2\leq\psi_1<\mu_2\leq\mu_1,\quad \psi_2<\mu_2\leq \mu_1<\psi_1.
$$
As $\Omega_{d_1,d_2}(1-\delta)< \Xi_{d_1,d_2}(1-\delta)$, increasing $A$ to $\Omega_{d_1,d_2}(1-\delta)$ makes $\mu_2=\mu_1$ while $\psi_2\leq \psi_1$ still exist. Hence, the first three cases violate Proposition \ref{prop_resultant py qy}. Since $P_X(1)=0$, we have $\omega(A,1,l)=-Q_X(1)P_Y(A,1,l)$ and $\mu_i(A,1)=\psi_i(A,1)$. Therefore, the fourth and the fifth cases also contradict Proposition \ref{prop_resultant py qy} by increasing $k$ from $1-\delta$ to $1$. Only the last case stays valid.
\end{proof}

Finally, we consider $\mathcal{E}^+_0\cap \{P=0\}$.
\begin{proposition}
\label{prop_Q<0 on E0}
Let $d_1,d_1\geq 2$ with $d_1+d_2\geq 6$ and $(d_1,d_2)\neq (2,4)$. If $A\in \left[\Psi_{d_1,d_2},\frac{1}{n+d_1}\frac{d_2(d_2-1)^2}{4(d_1-1)}\right)$, then $\omega<0$ on $\mathcal{E}^+_0\cap \{P=0\}$.
\end{proposition}
\begin{proof}
Suppose for some $\tilde{A}\in \left[\Psi_{d_1,d_2},\frac{1}{n+d_1}\frac{d_2(d_2-1)^2}{4(d_1-1)}\right)$ and $\tilde{k}\in(0,1)$, the polynomial $\omega(\tilde{A},\tilde{k},l)$ is positive for some $l$. From Proposition \ref{prop_omega2} and Proposition \ref{prop_omega0}, the quadratic function $\omega(\tilde{A},\tilde{k},l)$ has two positive roots $\mu_2(\tilde{A},\tilde{k})\leq \mu_1(\tilde{A},\tilde{k})$. By Proposition \ref{prop_prop of Omega 2}, the function $\omega(\tilde{A}, k ,l)$ has two real roots for any $k\in [\tilde{k},1)$. Furthermore, by Proposition \ref{prop_initial inclusion}, there is a $\delta>0$ small enough so that 
\begin{equation}
\label{eqn_inclusion 1-delta}
[\mu_2(\tilde{A},1-\delta), \mu_1(\tilde{A},1-\delta)]\subset (\psi_2(\tilde{A},1-\delta), \psi_1(\tilde{A},1-\delta)).
\end{equation}
Suppose $P_Y(\tilde{A},\tilde{k},l)$ does not have real roots, then Proposition \ref{prop_resultant py qy} is violated by $\omega(\tilde{A},k,l)$ and $P_Y(\tilde{A},k,l)$ for some $k\in (\tilde{k},1-\delta)$. Therefore, the polynomial $P_Y(\tilde{A},\tilde{k},l)$ also has real roots $\psi_2(\tilde{A},\tilde{k})\leq \psi_1(\tilde{A},\tilde{k})$. By Proposition \ref{prop_resultant py qy} and \eqref{eqn_inclusion 1-delta}, we also have the inclusion
$$[\mu_2(\tilde{A},\tilde{k}), \mu_1(\tilde{A},\tilde{k})]\subset (\psi_2(\tilde{A},\tilde{k}), \psi_1(\tilde{A},\tilde{k})).$$ Equivalently, the inequality $\omega(\tilde{A},\tilde{k},l)< 0$ holds if $P_Y(\tilde{A},\tilde{k},l)\leq 0$. As $P_X>0$ on $(0,1)$, we must have $P_Y<0$ from $P=0$. By the arbitrariness of $\tilde{A}$ and $\tilde{k}$, we conclude that $\omega$ is negative on $\mathcal{E}^+_0\cap \{P=0\}$.
\end{proof}

\begin{lemma}
\label{lem_gamma in E0 is clear}
For any $d_1,d_1\geq 2$ with $d_1+d_2\geq 6$, $(d_1,d_2)\neq (2,4)$ and $A\in \left[\Psi_{d_1,d_2},\frac{1}{n+d_1}\frac{d_2(d_2-1)^2}{4(d_1-1)}\right)$, an integral curve in $\mathfrak{C}_A$ can only escape the set through the faces $\{H=0\}\cap\{X_1>0\}$ or $\{Z-\mu_1 Y=0\}\cap \{X_1-X_2>0\}$.
\end{lemma}
\begin{proof}
From Proposition \ref{prop_P<0 on E2}, Proposition \ref{prop_Q<0 on E1} and Proposition \ref{prop_Q<0 on E0}, an integral curve in $\mathfrak{C}_A$ does not escape the set through the face $\left(\mathcal{E}^+\right)^\circ\cap \{P=0\}$. An integral curve in $\mathfrak{C}_A$ does not intersect $\{Y=0\}$ or $\{Z=0\}$ since these sets are invariant. Suppose an integral curve in $\mathfrak{C}_A$ escapes through the face $\{X_1-X_2=0\}\cap \{X_1>0\}$. By $P\leq 0$, the inequality $(X_1-X_2)'\geq 0$ holds at the intersection point. This is a contradiction as $X_1-X_2>0$ holds while the integral curve is in $\mathfrak{C}_A$. Therefore, an integral curve in $\mathfrak{C}_A$ can only escape the set through $\{H=0\}$ or $\{Z-\mu_1 Y=0\}\cap \{X_1-X_2>0\}$.

We claim that it is impossible that an integral curve in $\mathfrak{C}_A$ intersects 
$$\Gamma:= \mathcal{E}\cap \{X_1=X_2=0\}.$$
Suppose otherwise, the integral curve must intersect $\Gamma\cap \{R_1-R_2<0\}$. Both $P$ and $Q$ vanish at $\Gamma$ apparently. Straightforward computations show that $\left.P'\right|_{\Gamma}=\left.P''\right|_{\Gamma}=0$. Let the integral curve intersects $\Gamma\cap \{R_1-R_2<0\}$ along $kX_1=X_2$ for some $k\in \left[-\frac{d_1}{d_2},1\right]$. Then $k\left(R_1-\frac{1}{n}\right)=R_2-\frac{1}{n}$ at the intersection point. From \eqref{eqn_P'} and \eqref{eqn_P and Q in X and Y}, we obtain
\begin{equation}
\begin{split}
\left.P'''\right|_{\Gamma}&=2\left.\left((X_1-X_2)'Q'\right)\right|_{\Gamma}=2\left.\left((R_1-R_2)Q'\right)\right|_{\Gamma},\\
\left.Q'\right|_{\Gamma}&=Q_Y(A,k,l)Y^2\left(R_1-\frac{1}{n}\right).
\end{split}
\end{equation}

We show that $\left.P'''\right|_{\Gamma\cap \{R_1-R_2<0\}}<0$ for $k\in \left[-\frac{d_1}{d_2},1\right]$ in the following. By the definition of $\mathcal{E}$, we have
\begin{equation}
\label{eqn_sum of secitonal curvature}
\begin{split}
n\left(R_1-\frac{1}{n}\right)\leq d_1\left(R_1-\frac{1}{n}\right)+d_2\left(R_2-\frac{1}{n}\right)=-\frac{1}{n}\\
\end{split}
\end{equation}
on $\Gamma$. The inequality $R_1-\frac{1}{n}<0$ holds on $\Gamma\cap \{R_1-R_2<0\}$. Therefore, the inequality $\left.P'''\right|_{\Gamma\cap \{R_1-R_2<0\}}<0$ holds if and only if $Q_Y(A,k,l)<0$ for any $(A,l)\in \left[\Psi_{d_1,d_2},\frac{1}{n+d_1}\frac{d_2(d_2-1)^2}{4(d_1-1)}\right)\times (\mu_2,\mu_1)$. Since the function $Q_Y(A,k,l)$ is linear in $k$ (see \eqref{eqn_QY}), it suffices to check the negativity of $Q_Y\left(A,-\frac{d_1}{d_2},l\right)$ and $Q_Y\left(A,1,l\right)$. 

By the formula of $Q_Y(A,1,l)$, the quadratic function always has two roots with opposite signs. On the other hand, we have $P_Y(A,1,l)Y^2=R_2-R_1$. If $A=\frac{1}{n+d_1}\frac{d_2(d_2-1)^2}{4(d_1-1)}$, then $\mu_1=\mu_2=\frac{2(d_1-1)}{d_2-1}$. Given $(d_1,d_2)\notin\{(2,2),(2,3),(2,4)\}$, we have
$$
Q_Y\left(\frac{1}{n+d_1}\frac{d_2(d_2-1)^2}{4(d_1-1)},1,\frac{2(d_1-1)}{d_2-1}\right)=\frac{-2(d1 - 1)(2d_1n-3n-3d_1)n}{(n+d_1)(n - 1)d_1}<0.
$$
If $A=\Psi_{d_1,d_2}$, the positive root of $Q_Y(\Psi_{d_1,d_2},1,l)$ is $\mu_2$ the smaller root of $P_Y(\Psi_{d_1,d_2},1,l)$. Proposition \ref{prop_resultant py qy} and \eqref{eqn_res_omega_PY} implies that $Q_Y(A,1,l)$ and $P_Y(A,1,l)$ do not share any common root if $A> \Psi_{d_1,d_2}$. Therefore, we have $Q_Y(A,1,l)<0$ for any $(A,l)\in \left[\Psi_{d_1,d_2},\frac{1}{n+d_1}\frac{d_2(d_2-1)^2}{4(d_1-1)}\right)\times (\mu_2,\mu_1)$. On the other hand, by \eqref{eqn_conservation equation for 1} and \eqref{eqn_function Q},
\begin{equation}
\label{eqn_Q_Y at k=-d1/d2}
\begin{split}
Q_Y\left(A,-\frac{d_1}{d_2},l\right)Y^2
%&=-\frac{n+d_1}{d_2n}+(d_2-1)YZ+\frac{2d_1(d_1-1)Y^2}{d_2}\\
%&=-\frac{n+d_1}{d_2n}+\frac{n+d_1}{nd_2}\left(d_1R_1+d_2R_2\right)+\frac{d_1}{n}(R_1-R_2)\\
&=-\frac{n+d_1}{n^2d_2}+\frac{d_1}{n}(R_1-R_2)<0.
\end{split}
\end{equation}
We conclude that $\left.P'''\right|_{\Gamma\cap \{R_1-R_2<0\}}<0$. This is a contradiction as $P<0$ while the integral curve is in $\mathfrak{C}_A$.
\end{proof}

We are ready to prove Theorem \ref{thm_main}.
\begin{proof}[Proof for Theorem \ref{thm_main}]
By \cite[Proposition 4.6]{chi2024existence}, the function $P$ is initially negative along $\gamma_s$ if $\omega(A,k,l)<0$ along the $(k,l)$-projection of $\gamma_s$. Since $A\geq \Psi_{d_1,d_2}>\frac{d_2(d_2-1)^2}{d_1^2(d_1d_2-d_2+4)}$, the function $\omega(A,k,l)$ is negative in a neighborhood around $(0,0)$ in $(k_*,1)\times (0,\mu_1)$ by the second partial test. Therefore, each $\gamma_s$ with $s>0$ is initially in $\mathfrak{C}_A$.

It is obvious that $p_0^-\notin \overline{\mathfrak{C}_A}$. Suppose $\lim\limits_{\eta\to\infty}\gamma_{s_*} = p_0^-$ for some $s_*>0$. The integral curve has to escape $\mathfrak{C}_A$. By Proposition \ref{prop_l gets too large} we know that $\gamma_{s_*}$ does not intersect $\{Z-\mu_1Y=0\}\cap \{X_1-X_2>0\}$. By Lemma \ref{lem_gamma in E0 is clear}, the integral curve $\gamma_{s_*}$ escapes $\mathfrak{C}_A$ through $\{H=0\}\cap \{X_1>0\}$ and eventually intersects $\{X_1-X_2=0\}\cap \{H<0\}$. By the $\mathbb{Z}_2$-symmetry of the dynamical system, there exists a parameter $\bar{s}_*>0$ such that $\gamma_{\bar{s}_*}$ escape $\mathfrak{C}_A$ through $\{X_1-X_2=0\}\cap \{H>0\}$. This is a contradiction to Lemma \ref{lem_gamma in E0 is clear}.
\end{proof}

Our analysis also yields a more precise description of the distinguished solution $\gamma_{s_\star}$ appearing in
the existence theorem of \cite{chi2024existence}. 
\begin{proposition}
The integral curve $\gamma_{s_\star}$ initially lies in $\mathcal{E}_0^+$ near $p_0^+$ and remains in the set until its first intersection with $\Gamma$. Moreover,
$$
s_\star\geq \frac{d_1}{d_2-1}.
$$
\end{proposition}
\begin{proof}
We first show that no integral curve in $(\mathcal{E}_1^+\cup\mathcal{E}_2^+)\cap \mathfrak{C}_A$ can intersect $\Gamma$.

By Proposition \ref{prop_P<0 on E2} and Proposition \ref{prop_Q<0 on E1}, an integral curve in $(\mathcal{E}_1^+\cup\mathcal{E}_2^+)\cap \mathfrak{C}_A$ does not intersect $\mathcal{E}^+_{1}\cap \{P=0\}$. Hence, the function $P$ remains negative as long as $\gamma_s$ is in $\mathcal{E}_1^+\cup\mathcal{E}_2^+$. The integral curve does not enter $\mathcal{E}_0^+\cap \mathfrak{C}_A$, since $X_2\leq 0$ in $\mathcal{E}_1^+\cup\mathcal{E}_2^+$ and
$$X_2'=\left(R_2-\frac{1-H^2}{n}\right)=\frac{P}{X_1}<0$$
at $\{X_2=0\}\cap \mathfrak{C}_A$.
Therefore, an integral curve in $(\mathcal{E}_1^+\cup\mathcal{E}_2^+)\cap \mathfrak{C}_A$ either remains in the set, or escapes through the face $\{H=0\}\cap\{X_1\geq X_2\}\cap \{P<0\}$.

Suppose an integral curve in $(\mathcal{E}_1^+\cup\mathcal{E}_2^+)\cap \mathfrak{C}_A$ intersects $\Gamma$, the function $X_2$ increases to zero while $X_1$ decreases to zero. Hence, the integral curve must intersect $\Gamma\cap \{R_1-R_2<0\}$ along $kX_1=X_2$ with $k\in \left[-\frac{d_1}{d_2},0\right]$ and $R_2-\frac{1}{n}=k\left(R_1-\frac{1}{n}\right)\geq 0$. Then \eqref{eqn_Q_Y at k=-d1/d2} and 
$$
Q_Y(A,0,l)Y^2=-2(R_2-\frac{1}{n})-4\frac{d_1}{d_2}AZ_2<0
$$
imply that $P'''<0$ at the intersection point, a contradiction.

By \cite[Proposition 4.11]{chi2024existence} and \cite[(4.21)]{chi2024existence}, if a $\gamma_s$ is in $\mathcal{E}_1^+$ initially, the integral curve is initially in $\mathcal{E}_1^+\cap \mathfrak{C}_A$. Since the distinguished trajectory $\gamma_{s_\star}$ intersects $\Gamma$, it must initially lie in $\mathcal{E}_0^+$. We have $s_\star\geq \frac{d_1}{d_2-1}$ by the linearized solution in \cite[(3.2)]{chi2024existence}

Suppose $\gamma_{s_\star}$ enters $\mathcal{E}^+_1$ through the face $\{X_2=0\}$. We have
$$
X_2'=R_2-\frac{1}{n}\leq 0,\quad  P=X_1\left(R_2-\frac{1}{n}\right)\leq 0
$$
at the intersection point. If $X_2'<0$, the integral curve enters $\mathcal E_1^+\cap \mathfrak{C}_A$, and hence cannot intersect $\Gamma$ by the preceding argument. If $X_2'=0$, then $P=0$. Since $\omega<0$ at the intersection point by Proposition \ref{prop_Q<0 on E1}, the integral curve enters $\mathcal{E}_1^+\cap \mathfrak{C}_A$, and again cannot intersect $\Gamma$. This is a contradiction in either case. Therefore, $\gamma_{s_\star}$ remains in $\mathcal{E}_0^+$ until its first intersection with $\Gamma$.
\end{proof}

\section{Appendix}
We first claim the following proposition, which is useful in determining the sign of some polynomials.

\begin{proposition}
\label{prop_useful polynomial prop}
For a cubic polynomial $p(x)=p_3x^3+p_2x^2+p_1x+p_0$ is positive non-negative on $[0,1]$ if $\sum_{i=0}^jp_i\geq 0$ for $j=0,1,2,3$.
\end{proposition}
\begin{proof}
Rewrite $p(x)$ as 
$$
p(x)=(p_3+p_2+p_1+p_0)x^3+(p_2+p_1+p_0)(x^2-x^3)+(p_1+p_0)(x-x^2)+p_0(1-x).
$$
The statement then follows.
\end{proof}

\begin{proposition}
\label{prop_omega2}
For any $d_1, d_2\geq 2$, we have $\omega_2<0$ for each $k\in\left[-\frac{d_1}{d_2},1\right]$.
\end{proposition}
\begin{proof}
By \cite[Proposition 4.13]{chi2024existence}, the inequality holds on $[0,1]$.
For $\omega_2(k)$ on $\left[-\frac{d_1}{d_2},0\right)$, consider $k=-\frac{d_1}{d_2}(1-\kappa)$. We have
\begin{equation}
\begin{split}
\omega_2\left(-\frac{d_1}{d_2}\kappa\right)=-\frac{d_1^3}{d_2}\tilde{\omega}_2(\kappa),
\end{split}
\end{equation}
where 
\begin{equation}
\begin{split}
&\tilde{\omega}_2(\kappa)\\
&= (-2d_1^2 + d_1d_2 - d_2 + 1)\kappa^3 + (2d_1^2 + d_1d_2 + 2d_1 - d_2 - 1)\kappa^2 - (2d_1 + 1)\kappa+1\\
&=(-d_1^2 + d_1d_2 - d_2 + 1)\kappa^3 + (d_1^2 + d_1d_2 - d_2 - 1)\kappa^2+(1-\kappa)(1-d_1\kappa)^2\\
&\geq (-d_1^2 + d_1d_2 - d_2 + 1)\kappa^3 + (d_1^2 + d_1d_2 - d_2 - 1)\kappa^3\\
&\geq 0
\end{split}
\end{equation}
Hence, the inequality $\omega_2(k)<0$ holds on $\left[-\frac{d_1}{d_2},1\right]$ for any $d_1,d_2\geq 2$.
\end{proof}

\begin{proposition}
\label{prop_omega0}
For any $d_1,d_2 \geq 2$, we have  $\omega_0> 0$ for each $k\in\left[-\frac{d_1}{d_2},1\right]$.
\end{proposition}
\begin{proof}
Since 
$$ \frac{d\omega_0}{dk}\left(-\frac{d_1}{d_2}\right)=2d_1d_2(d_1+1)>0,$$ the function $\omega_0$ either monotonically increases on $\left[-\frac{d_1}{d_2},1\right]$, or it increases and then decreases on the interval. 
Since 
$$\omega_0 \left(-\frac{d_1}{d_2}\right)=\frac{d_1^2(n+d_1)}{d_2}>0,\quad \omega_0(1)=d_2(n-1)(d_1n-n-d_1)>0,$$
the positivity of $\omega_0$ on $\left[-\frac{d_1}{d_2},1\right]$ is established.
\end{proof}

\begin{proposition}
\label{prop_PX+QX is negative}
The polynomial $P_X+Q_X$ is negative on $\left[-\frac{d_1}{d_2},1\right]$
\end{proposition}
\begin{proof}
By \cite[Proposition B.1]{chi2024existence}, we have $P_X+Q_X<0$ on $[0,1]$. We proceed to consider the function on $\left[-\frac{d_1}{d_2},0\right]$. Consider $k=-\frac{d_1}{d_2}\kappa$ and $P_X+Q_X$ as a polynomial on $\kappa\in [0,1]$. We have 
\begin{equation}
\begin{split}
-\frac{d_2^2(n-1)}{d_1}(P_X+Q_X)&=\Pi(\kappa)\\
&=(- d_1^2  d_2 -  d_1  d_2^2 + 4  d_1^2 + 3  d_1  d_2 +  d_2^2 -  d_1 -  d_2) \kappa^3\\
&\quad  + ( d_1  d_2^2 - 2  d_1^2 + 2  d_1  d_2 -  d_2^2 + 2  d_1 +  d_2) \kappa^2 \\
&\quad + ( d_1^2  d_2 -  d_1  d_2^2 - 3  d_1  d_2 +  d_2^2 + 2  d_2) \kappa +  d_1  d_2^2 -  d_2^2
%&= (2d_1^2+2d_1d_2+d_1+2d_2) \kappa^3\\
%&\quad  + (d_1d_2^2-2d_1^2+2d_1d_2-d_2^2+2d_1+d_2) (\kappa^2-\kappa^3) \\
%&\quad + ( d_1^2  d_2  - 3  d_1  d_2  + 2  d_2) (\kappa-\kappa^3) +  (d_1  d_2^2 -  d_2^2)(1-\kappa)\\
%&\geq (2d_1^2+2d_1d_2+d_1+2d_2) \kappa^3\\
%&\quad  + (d_1d_2^2-2d_1^2+2d_1d_2-d_2^2+2d_1+d_2) (\kappa^2-\kappa^3) \\
%&\quad + ( d_1^2  d_2  - 3  d_1  d_2  + 2  d_2) (\kappa^2-\kappa^3) +  (d_1  d_2^2 -  d_2^2)(1-\kappa)\\
%&= (2d_1^2+2d_1d_2+d_1+2d_2) \kappa^3\\
%&\quad  + (d_1d_2^2+d_1^2  d_2-2d_1^2-d_1d_2-d_2^2+2d_1+3d_2) (\kappa^2-\kappa^3) \\
%&\quad  +  (d_1  d_2^2 -  d_2^2)(1-\kappa)
\end{split}
\end{equation}
We have 
\begin{align*}
\Pi_{0}&=d_2^2(d_1-1)>0\\
\Pi_{1}+\Pi_{0}&=d_2(d_1 - 1)(d_1-2)\geq 0\\
\Pi_2+\Pi_{1}+\Pi_{0}&=d_1^2 (d_2-2) + d_2(d_1 d_2 - d_1 - d_2) + 2d_1 + 3d_2>0\\
\Pi_3+\Pi_2+\Pi_{1}+\Pi_{0}&=2d_1^2 + 2d_1d_2 + d_1 + 2d_2>0\\
\end{align*}
By Proposition \ref{prop_useful polynomial prop}, we have $P_X+Q_X< 0$ on $\left[-\frac{d_1}{d_2},1\right]$. The proof is complete.
\end{proof}

\begin{proposition}
\label{prop_property of Omega}
On the interval $(0,1)$, the function $\Omega_{d_1,d_2}$
\begin{enumerate}
\item
increases then decreases if $(d_1,d_2)=(2,2)$;
\item
increases, then decreases, then increases if $(d_1,d_2)=(3,2)$;
\item
decreases then increases if $d_1=2,3$, $d_2\geq 3$;
\item
monotonically increases if $d_1\geq 4$, $d_2\geq 2$.
\end{enumerate}
\end{proposition}
\begin{proof}
Compute the derivative  we obtain
\begin{equation}
\frac{d\Omega_{d_1,d_2}}{dk}=\frac{d_2^2(d_2-1)^2}{(2d_1+d_2k)^2(d_1-1)\omega_0^2\omega_2^2}\alpha_1\alpha_2\alpha_3\alpha_4,
\end{equation}
where
\begin{equation}
\begin{split}
\alpha_1&=(d_1d_2^3 - 4d_1d_2^2 - d_2^3 + 3d_2^2)k^3 + (2d_1^2d_2^2 - 8d_1^2d_2 + 8d_1d_2 - 2d_2^2)k^2\\
&\quad  + (d_1^3d_2 - 4d_1^3 + 5d_1^2d_2 + 4d_1^2 - 6d_1d_2)k + 4d_1^3 - 4d_1^2,\\
\alpha_2&=d_2(d_1-1)(d_2-1)k^2-d_1d_2k^2 + d_1(2d_1 d_2 - 2d_1 - 2 d_2 + 2)k + d_1^3 - d_1^2\\
%&=(d_1d_2^2 - 2d_1d_2 - d_2^2 + d_2)k^2 + (2d_1^2d_2 - 2d_1^2 - 2d_1d_2 + 2d_1)k + d_1^3 - d_1^2,\\
\alpha_3&=(d_1d_2^2 + d_1d_2 - d_2^2)k^2 + (2d_1^2d_2 + 2d_1^2 - 4d_1d_2)k + d_1^3 - 4d_1^2,\\
\alpha_4&=(-d_1^3d_2^3 + d_1^2d_2^4 + 4d_1^3d_2^2 - 2d_1^2d_2^3 - 2d_1d_2^4 - d_1^2d_2^2 + 4d_1d_2^3 + d_2^4 - d_1d_2^2 - d_2^3)k^3\\
&\quad + (-2d_1^4d_2^2 + 4d_1^3d_2^3 + 8d_1^4d_2 - 8d_1^3d_2^2 - 8d_1^2d_2^3 - 8d_1^3d_2 + 14d_1^2d_2^2 + 4d_1d_2^3 - 4d_1d_2^2)k^2\\
&\quad + (-d_1^5d_2 + 5d_1^4d_2^2 + 4d_1^5 - 10d_1^4d_2 - 10d_1^3d_2^2 - 8d_1^4 + 17d_1^3d_2 + 5d_1^2d_2^2 + 4d_1^3 - 6d_1^2d_2)k \\
&\quad + 2d_1^5d_2 - 4d_1^5 - 4d_1^4d_2 + 8d_1^4 + 2d_1^3d_2 - 4d_1^3.\\
\end{split}
\end{equation}

We consider the sign of each factor in the following.
\begin{enumerate}
\item
$\alpha_1$\\
We have 
\begin{align*}
\alpha_{10}&=4d_1^3-4d_1^2>0\\
\alpha_{11}+\alpha_{10}&=d_1(d_1 - 1)(4d_1^2 + d_1d_2 + 4d_1 + 6d_2)>0\\
\alpha_{12}+\alpha_{11}+\alpha_{10}&=(d_1 - 1) (4 d_1^3 + d_1^2 d_2 + 2 d_1 d_2^2 + 4 d_1^2 - 2 d_1 d_2 + 2 d_2^2)>0\\
\alpha_{13}+\alpha_{12}+\alpha_{11}+\alpha_{10}&=4  d_1^4 +  d_1^3  d_2 + 2  d_1^2  d_2^2 +  d_1  d_2^3 - 3  d_1^2  d_2 - 4  d_1  d_2^2 -  d_2^3 - 4  d_1^2 + 2  d_1  d_2 +  d_2^2\\
&=4  (d_1^4 -d_1^2)+ d_1d_2(d_1^2-3d_1+2) + 2  (d_1-2)  d_1d_2^2 +  (d_1 -1) d_2^3  +  d_2^2\\
&>0\\
\end{align*}
By Proposition \ref{prop_useful polynomial prop}, the factor $\alpha_1$ is positive on $[0,1]$ for all $d_1,d_2\geq 2$.
\item
$\alpha_2$\\
For $\alpha_2$, we have
\begin{equation}
\label{eqn_beta_2a>0}
\begin{split}
\alpha_2&\geq -d_1d_2k^2 + d_1(2d_1 d_2 - 2d_1 - 2 d_2 + 2)k + d_1^3 - d_1^2\\
&>  d_1(2d_1 d_2 - 2d_1 - 3 d_2 + 2)k\\
&\geq 0
\end{split}
\end{equation}
for any $d_1,d_2\geq 2$.
\item
$\alpha_3$\\
The coefficients of $k^2$ and $k$ in $\alpha_3$ are positive for any $d_1,d_2\geq 2$. Therefore, we have $\alpha_3\geq 0$ if $d_1\geq 4$. 
 If $d_1\in \{2,3\}$, we have 
\begin{align*}
d_1=2:&\quad \alpha_3=(d_2^2 + 2d_2)k^2 + 8k -8,\\
d_1=3:&\quad \alpha_3=(2d_2^2 + 3d_2)k^2 + (6 d_2 +18)k -9,
\end{align*}
which both have a real root in $(0,1)$ for any $d_2\geq 2$.

\item
$\alpha_4$\\
We have 
\begin{align*}
\alpha_{40}&=2d_1^3(d_1-1)^2(d_2-2)\\
\alpha_{41}+\alpha_{40}&= d_1^2  d_2 ( d_1 - 1) ( d_1^2 + 5  d_1  d_2 - 13  d_1 - 5  d_2 + 6)\\
&= d_1^2  d_2 ( d_1 - 1) ( (d_1-2)(d_1+5d_2-12)+d_1+5d_2-18)\\
\alpha_{42}+\alpha_{41}+\alpha_{40}&= d_1  d_2 ( d_1 - 1) ( d_1^3 + 3  d_1^2  d_2 + 4  d_1  d_2^2 - 5  d_1^2 - 15  d_1  d_2 - 4  d_2^2 + 6  d_1 + 4  d_2)\\
&=d_1  d_2 ( d_1 - 1) ((4d_1 - 4)d_2^2 + (3d_1^2 - 15d_1 + 4)d_2 + d_1(d_1-2)(d_1-3))\\
\alpha_{43}+\alpha_{42}+\alpha_{41}+\alpha_{40}&= d_2( d_2+d_2 - 1 )( d_1^2 +  d_1 d_2 - 3 d_1 -  d_2)( d_1^2 +  d_1 d_2 - 2 d_1 -  d_2)\\
\end{align*}
It is evident that the first and the last quantities are non-negative for any $d_1,d_2\geq 2$. The quantity $\alpha_{41}+\alpha_{40}$ is negative only for $(d_1,d_2)\in \{(2,2),(2,3),(3,2)\}$, while the quantity $\alpha_{42}+\alpha_{41}+\alpha_{40}$ is negative only for $(d_1,d_2)\in \{(2,2),(2,3)\}$.
By Proposition \ref{prop_useful polynomial prop}, the factor $\alpha_4$ is non-negative on $[0,1]$ if $(d_1,d_2)\notin \{(2,2),(2,3),(3,2)\}$. Since
\begin{align*}
(d_1,d_2)=(2,2):& \quad \alpha_4=48k^3-48k\leq 0\\
(d_1,d_2)=(2,3):& \quad \alpha_4=72k^3-24k^2-28k+16>0\\
(d_1,d_2)=(3,2):& \quad \alpha_4=176k^3 + 192k^2 - 144k
\end{align*}
The factor is non-negative for all $d_1,d_2\geq 2$ with $(d_1,d_2)\notin\{(2,2),(3,2)\}$.
\end{enumerate}
\end{proof}

%\begin{table}[H]
%\centering
%\begin{tabular}{|l|l|l|l|l|l|l|l|}
%\hline
%$\mathsf{K}$&$\mathsf{H}$&$\mathsf{G}$& $d_1$& $d_2$& $\alpha$&$A$&$\Psi_{d_1,d_2}$\\
%\hline
%\hline
%&&&&&&&\\
%$Sp(2)U(1)$& $U(4)$ & $SU(5)$ & $5$ & $8$ & $\frac{4}{5}$&$\frac{49}{50}=0.98$& $\frac{186494}{198025}\approx 0.94$\\[2ex]
%
%$\spin(7)$ & $\spin(8)$ & $\spin(9)$ &$7$ & $8$&$\frac{6}{7}$&$\frac{1}{2}$& $\frac{8879}{20886}\approx 0.43$\\[2ex]
%
%$G_2\times SO(2)$ & $\spin(7)SO(2)$ &$\spin(9)$& $7$ & $14$ & $\frac{5}{7}$&$\frac{507}{196}\approx 2.59$& $\frac{11}{6}\approx 1.83$\\[2ex]
%
%$\spin(11)Sp(1)$ & $\spin(12)Sp(1)$ & $E_7$ & $11$ & $64$ & $\frac{5}{9}$& $49$ & $\frac{26823819708}{1214772845}\approx 22.08$ \\[2ex]
%
%$\spin(15)$ & $\spin(16)$ & $E_8$& $15$ & $128$ & $\frac{7}{15}$ & $\frac{32258}{225}\approx 143.37$ &$\frac{28882022881}{576131150}\approx 50.13$\\[2ex]
%\hline
%\end{tabular}
%\caption{}
%\label{table: Non-existence}
%\end{table}

\textbf{Acknowledgement.}
The author would like to thank the conference organizers of ``Einstein Spaces and Special Geometry'' at the Institut Mittag-Leffler in July 2023, where part of this work was presented. The author thanks Christoph B\"ohm, Andrew Dancer, Lorenzo Foscolo, and Claude LeBrun for insightful discussions during the conference. The author also thanks McKenzie Wang and Xiping Zhu for valuable comments on an early draft of this paper.

\bibliography{BIB}
\bibliographystyle{alpha}
\end{document}